\documentclass[11pt,draft]{amsart}
\usepackage[pctex32]{graphics}
\usepackage{amsfonts}
\usepackage{amssymb,amscd,latexsym}
\usepackage{amsmath}
\usepackage{epsfig}
\usepackage{color}
\usepackage{color}
\usepackage{comment}

\usepackage[utf8]{inputenc}
\usepackage[english]{babel} 
\usepackage{enumerate}
\usepackage{verbatim}

\newtheorem{thm}{Theorem}[section]
\newtheorem{cor}[thm]{Corollary}
\newtheorem{prop}[thm]{Proposition}

\newtheorem{defin}[thm]{Definition}

\newtheorem{lema}[thm]{Lemma}

\newtheorem{Remark}[thm]{Remark}
\theoremstyle{definition}
\newtheorem{ex}{Example}

\newcommand{\PP}{{\mathbb{P}}}

\def\p{\mathbb P}

\def\Z{\mathbb{Z}}

\def\P{\mathbb{P}}

\def\rk{\operatorname{rk}}

\begin{document}

\title{Moduli of Cubic fourfolds and reducible OADP surfaces}

\author[Michele Bolognesi]{Michele Bolognesi}
\address{Institut Montpellierain Alexander Grothendieck \\ %
Universit\'e de Montpellier \\ %
CNRS \\ %
Case Courrier 051 - Place Eug\`ene Bataillon \\ %
34095 Montpellier Cedex 5 \\ %
France}
\email{michele.bolognesi@umontpellier.fr}
\thanks{M.B. is supported by the ANR project FanoHK (ANR-20-CE40-0023). MB is a member of the Réseau thématique \it Géométrie algébrique et singularités, \rm and of the research group \it  GNSAGA \rm}

\author[Z.Brahimi]{Zakaria Brahimi}
\address{Dipartimento di Matematica e Fisica, Universit\'a Roma Tre, Largo San Leonardo Murialdo  00146 Roma
Italy} 
 \email{zakaria.brahimi@uniroma3.it}

\author[H.Awada]{Hanine Awada}
\address{BCAM, Basque Center for Applied Mathematics, Mazarredo 14, 48009 Bilbao,
Basque Country, Spain}
\email{hawada@bcamath.org}
\thanks{H.A. was supported by the Spanish Ministry of Science through the Severo Ochoa Grant SEV
2023-2026 and through the research project PID2020-114750GB-C33 and by the Basque
Government through the BERC 2022-2025 program.}

\maketitle

\begin{abstract}
In this paper we explore the intersection of the Hassett divisor $\mathcal C_8$, parametrizing smooth cubic fourfolds $X$ containing a plane $P$ with other divisors $\mathcal C_i$. Notably we study the irreducible components of the intersections with $\mathcal{C}_{12}$ and $\mathcal{C}_{20}$. These two divisors generically parametrize respectively cubics containing a smooth cubic scroll, and a smooth Veronese surface. First, we find all the irreducible components of the two intersections, and describe the geometry of the generic elements in terms of the intersection of $P$ with the other surface. Then we consider the problem of rationality of cubics in these components, either by finding rational sections of the quadric fibration induced by projection off $P$, or by finding examples of reducible one-apparent-double-point surfaces inside $X$. Finally, via some Macaulay computations, we give explicit equations for cubics in each component.

\end{abstract}

\section{Introduction}

Cubic hypersurfaces in $\PP^5$ are among the most studied, and at the same time the most mysterious objects in algebraic geometry. The reason is probably the wealth of geometry that they contain, and the fact that there are lots of very basic and classical problems that have not been cleared out yet, for example the rationality of the generic cubic fourfold. The study of their moduli space, particularly through GIT and the period map, has seen some very striking advances in recent years, see for example  \cite{voisin,laza,Loij}, and the study of rationality has been developed in parallel to this. In particular Hassett \cite{Ha2} has described a countable infinity of divisors $\mathcal C_d$ that parametrize \it special cubic 4-folds, \rm that is \rm cubic hypersurfaces  containing  a surface not homologous to a complete intersection. Indeed, the only examples of rational cubic fourfolds known so far are contained in some very specific Hassett divisors \cite{RS1,RS2,BRS,BD,ABBVA,Fano, awada,AHTVA}.

\medskip

In the first part of this paper, we consider a problem that has been addressed for the first time in \cite{ABBVA}. We consider the intersections of $\mathcal C_8$ with other Hassett divisors, notably $\mathcal C_{12}$ and $\mathcal C_{20}$. The general cubic fourfolds in these two divisors contain, respectively, a cubic scroll and a Veronese surface. By means of lattice theory, we describe all the irreducible components of $\mathcal{C}_8\cap \mathcal{C}_{20}$ and $\mathcal{C}_8\cap \mathcal{C}_{12}$.

Here are our results.
\begin{thm} (= Thm. \ref{cap812} and Lemma \ref{12coincide} ) There are three irreducible components of $\mathcal C_8\cap \mathcal C_{12}$ 	indexed by  the value $P\cdot S =\epsilon \in\{1, 2, 3\}$, where $P$ is a plane and $S$ the class of a cubic rational normal scroll (that is $S\cdot S=7$ and $S\cdot h^2=3)$ contained in the cubic fourfold. For $\epsilon=2$, every element in the corresponding irreducible component	is rational. For the general element of the irreducible component corresponding to $\epsilon=1$ the plane $P$ and the surface $S$ intersect at a point. 
\end{thm}
\begin{thm}(= Thm. \ref{compC208}) There are seven  irreducible components of $\mathcal C_8\cap \mathcal C_{20}$	indexed by  $P\cdot V=\gamma\in\{-2,-1, 0, 1, 2, 3, 4\}$, where $P$ is a plane and $V$ the class of a Veronese surface (that is $V\cdot V=12$ and $V\cdot h^2=4$) contained in the cubic fourfold. For $\gamma=-1,1$ and $3$ each smooth cubic hypersurface belonging to the corresponding irreducible component is rational.\end{thm}

As the reader can see, we also draw information about the rationality of cubics inside certain irreducible components. This is based principally on two different techniques. The first is now quite standard, and it was introduced by Hassett in \cite{Ha}. In fact, in certain cases we are able to find an algebraic 2-cycle on a smooth cubic fourfold $X$ that is a rational odd-degree multi-section of the quadric surface fibration obtained by projection off the plane $P\subset X$. This implies the rationality of the fibration, and of the cubic fourfold. The second technique is somehow less commonly used in the realm of cubic fourfolds, it is rooted on \cite{CR} and has been partially used in \cite{BRS}. Basically, by using the fact that $X$ contains a couple of independent surfaces, we manage to construct inside $X$ one-apparent-double-point surface (OADP in the following) obtained as the union of other surfaces that intersect in some loci. It is a classical fact (see for example \cite{CMR}) that a cubic fourfold containing an OADP surface is rational.

In particular, inside the components of $\mathcal{C}_{20}\cap \mathcal{C}_{8}$, we find cubic fourfolds containing interesting reducible degenerations of the Veronese surface, given by a cubic scroll plus a plane intersecting the scroll along the directrix line. These examples hence lie in $\mathcal{C}_{20}\cap \mathcal{C}_{8}\cap \mathcal{C}_{12}$. On the other hand, inside the component of $\mathcal{C}_8\cap\mathcal{C}_{12}$ where the cubic scroll has intersection 0 with the plane, we find cubics containing a plane intersecting the scroll along a line of the ruling. These reducible surfaces are degenerations of quartic scrolls, and the cubics containing such objects are hence contained in $\mathcal{C}_8\cap \mathcal{C}_{12}\cap \mathcal{C}_{14}$.


\subsection{Description of the contents}

In Section \ref{Prel} we resume shortly the information about algebraic cycles and the moduli space of cubic fourfolds. Then we give some generalities on OADP varieties and on the intersection theory of 2-cycles on cubic fourfolds. In Section \ref{c12}, we describe the irreducible components of $\mathcal{C}_{12}\cap \mathcal{C}_8$ and explain the geometry of the generic cubic fourfolds in the components, in particular the rational ones. The same results, this time for $\mathcal{C}_{20}\cap \mathcal{C}_8$, are explained in Section \ref{c20}. Finally, in the Appendix, we give some explicit examples of cubic fourfolds in the irreducible components outlined in Sections \ref{c12} and \ref{c20}, and explain shortly how we obtained them via Macaulay2 computations.

\subsection{Acknowledgements:} We wish to thank F.Russo, A.Verra and G.Staglian\`o for some fruitful and pleasant conversations.

\section{Preliminary results}\label{Prel}

\subsection{Hodge theory for cubic fourfolds}\label{section 2}
Let X be a cubic fourfold over $\mathcal{C}$, that is a smooth hypersurface of degree 3 in $\PP^5$. We will denote by $\mathcal{C}$ the moduli space of smooth cubic fourfolds in $\P^5$. It is a twenty-dimensional quasi-projective variety. The Hodge diamond of X is as follows

\begin{center}

1\\0  \hspace{0.5cm} 0\\ 0 \hspace{0.5cm} 1 \hspace{0.5cm} 0\\ 0 \hspace{0.5cm} 0 \hspace{0.5cm} 0 \hspace{0.5cm} 0 \\ 0 \hspace{0.5cm} 1 \hspace{0.5cm} 21 \hspace{0.5cm} 1 \hspace{0.5cm} 0

\end{center}

Let us now concentrate on the middle cohomology of X, that contains most nontrivial Hodge theoretic geometric information.
 
 Let $L$ be the cohomology group $H^4(X,\Z)$, which is also called \textit{cohomology lattice}, and $L_{prim} = H^4_{prim}(X,\Z) := \langle h^2 \rangle^{\perp}$ the \textit{primitive cohomology lattice}, where $h\in H^2(X,\Z)$ is the \textit{hyperplane class} defined by the embedding $X \subset \mathbb{P}^5$. We recall that $L_{prim}$ is an even lattice (see \cite[\S 2]{Ha2}).
 
 More precisely, we will consider the lattice of integral middle Hodge classes of $X$:
 
 $$M(X)= H^{2,2}(X) \cap H^4(X,\Z)= H^{2}(X, \Omega^2_X) \cap H^4(X,\Z),$$ 
 
 which comes equipped with the integer valued intersection form $-,- $. The Hodge-Riemann bilinear relations imply that $M(X)$ is a positive definite lattice. For cubic fourfolds, the (integral) Hodge conjecture holds (see  \cite[Theorem 18]{Voi}), and algebraic, rational and homological equivalences coincide for cycles of codimension 2. The upshot is that the cycle map $CH^2(X) \to H^4(X,\mathbb{Z})$ is injective, where $CH^2(X)$ denotes the Chow group of codimension 2 cycles on $X$ up to rational equivalence (see \cite[\S 5]{CT} or \cite{BP20}). Notably, by the Hodge conjecture, the algebraic cycles are the $(2,2)$-part of $H^4(X,\mathbb{Z})$. The generic cubic fourfold $X\in \mathcal{C}$ has $\rk{M(X)}=\rk \langle h^2 \rangle=1$, and a cubic fourfold is said to be \it special \rm whenever $\rk{M(X)}>1$. This is equivalent to saying that $X$ contains a surface that is not homologous to the 2-dimensional linear section. Let us call $d(M(X)) \in \Z$ the discriminant of the lattice $M(X)$. This is just the determinant of the Gram matrix. 

A \textit{labelling} of a special cubic fourfold is a positive definite rank two saturated sublattice $K_d$, with $h^2 \in K_d \subseteq M(X)$. The discriminant $d$ is the determinant of the matrix of the intersection form on $K_d$.

The loci of special cubic fourfolds form a countably infinite union of irreducible divisors $\mathcal{C}_d$ in $\mathcal{C}$ with a labelling of discriminant $d$, where $d$ takes certain integer values. 
  
Only very few $\mathcal{C}_d$'s can be defined explicitly in terms of particular surfaces contained in $X$. For example, $\mathcal{C}_8$ is the locus of cubic fourfolds containing a plane, $\mathcal{C}_{12}$ is defined as the closure of the one containing a cubic scroll, $\mathcal{C}_{14}$ as the closure of the locus of cubics containing a quartic scroll (or equivalently a quintic del del Pezzo surface),  $\mathcal{C}_{20}$ is the closure of the locus of cubic fourfolds containing  a Veronese surface, and a few others.

\subsection{The divisor $\mathcal{C}_8$ and the quadric surface fibration}

As we have stated here above, the divisor $\mathcal{C}_8$ is the locus of cubics $X\subset \PP^5$ containing a plane $P$, hence they have a labelling $K_8$ of this shape:

\begin{equation}\label{eq:disc8}
\begin{tabular}{|c|c|c|}
\hline     &  $h^2$ & P \\
\hline  $h^2$ & 3 & 1   \\
\hline   P & 1 & 3 \\  
\hline
\end{tabular},
\end{equation}

The linear projection $\pi_P:X\dashrightarrow \PP^2$ with center $P$, is resolved by blowing up $P$

$$\begin{aligned}
\widetilde{\pi}_P: Bl_{P}(X) \longrightarrow \mathbb{P}^{2}
\end{aligned}$$

into a morphism, that has naturally a quadric bundle structure. Under mild conditions on the plane (see \cite[Prop. 1.2.5]{ABB}), the quadric bundle degenerates on a smooth sextic curve $D\subset \PP^2$.

\begin{prop}\cite[Prop. 2.3]{Ha}\label{multisection}
Let $q: \mathcal Q \to B$ be a quadric surface bundle over a rational projective variety. Let $Q$ denote the class of the generic fiber of $q$ inside $H^4(\mathcal{Q}, \mathbb Z)\cap H^{2,2}(\mathcal Q)$. Assume there is a class $T \in H^4(\mathcal{Q}, \mathbb Z)\cap H^{2,2}(\mathcal Q)$ which has odd intersection with $Q$, then $\mathcal{Q}$ is rational over $\mathbb{C}$.
\end{prop}

In our situation, this is equivalent to saying that there is an odd degree multi-section of the quadric bundle $\widetilde{\pi}_P$, and we will use several times Prop. \ref{multisection} in the rest of the paper to check that certain classes of cubic fourfolds are rational.

\subsection{Generalized OADP varieties}

Let us recall the following important definition.
\begin{defin}
Let $X$ be an equidimensional reduced scheme in $\p^{2n+1}$ of dimension $n$. The scheme $X$ is called {\it a (generalized) variety with one apparent double point} if through a general point of $\p^{2n+1}$
there passes a unique secant line to $X$, that is a unique line cutting $X$ scheme theoretically in a reduced length two scheme.
\end{defin}

 The name $OADP$ variety is usually reserved for the irreducible reduced scheme satisfying the previous condition. This somehow bizarre name comes from the fact that the projection of $X$ from a general point into $\p^{2n}$ acquires a unique singular point, which is {\it double}, which means that its tangent cone is a reducible quadric of rank 2.
\medskip

We call \it abstract secant variety \rm $S_X$ of an irreducible variety $X\subset\p^{2n+1}$ the restriction of the universal family of the Grassmannian $\mathbb G(1,2n+1)$ to the closure of the image of the rational map which associates to a couple $(p,q)\in X\times X$ (with $p\neq q$) the line $\overline{pq}$ spanned by them. When $X$ is an OADP variety, the tautological morphism $p:S_X\to\p^{2n+1}$ is birational. This means that, by Zariski Main Theorem, the locus of points of $\PP^{2n+1}$ through which there passes
more than one secant line has codimension at least two inside $\PP^{2n+1}$. The same property holds true for reducible OADP varieties.
\medskip

\begin{cor}\cite[Cor 2.7]{BRS} Let $X\subset\p^N$ be  a non degenerate reduced algebraic set scheme theoretically defined by quadratic forms such that 
their Koszul syzygies are generated by linear syzygies. If through a general point of $\p^N$ there passes a finite number of secant lines
to $X$, then $X\subset\p^N$ is a generalized OADP variety.

In particular a small algebraic set $X\subset\p^N$ such that through a general point of $\p^N$ there passes
a finite number of secant lines to $X$ is a generalized  OADP variety.
\end{cor}

\subsection{Intersection theory on a cubic fourfold}

Let $X\subset\p^5$ be a smooth cubic 4-fold. Let $S\subset X$ be a smooth surface
and let $P\subset X$ be a plane s.t. the scheme theoretic intersection $S\cap P$ is a smooth curve $C$ of genus $g$ and degree $d$. Let us now compute the intersection number $S\cdot P\in\mathbb Z$ inside $X$ under the previous hypothesis. By \cite[Prop. 9.1.1, third formula]{Fu} we have
\begin{equation}\label{eq:Fulton}
S\cdot P=c_1(N_{S/X|C})\cdot c_1(T_{P|C})^{-1}\cdot c_1(T_C)\cap C.
\end{equation}

Let us denote by $K_Y$ the canonical class
of an arbitrary smooth projective variety $Y$. It is straightforward to see that $c_1(T_C)\cap C=2-2g(C)=-d(d-3)$ and $c_1(T_{P|C})^{-1}\cap C=K_P\cdot C=-3d$. By using the short exact sequence

$$0\to T_{S|C}\to T_{X|C}\to N_{S/X|C}\to 0,$$

we get the following equalities:

$$c_1(N_{S/X|C})\cap C=c_1(T_{X|C})\cap C+c_1(T_{S|C})^{-1}\cap C=-K_X\cdot C+K_S\cdot C=3d+K_S\cdot C.$$

By combining the previous calculation with \eqref{eq:Fulton} we deduce

\begin{equation}\label{eq:excess}
S\cdot P=K_S\cdot C-d(d-3)=K_S\cdot C+2-2g(C)=-\deg(N_{C/S}).
\end{equation}

More generally, always from \cite[Prop. 9.1.1]{Fu} (see \cite[Prop. 2.8]{BRS} for this particular form of the statement).

\begin{prop}\label{multiplicity}
Let $X \subset \PP^5$ be a smooth
cubic hypersurface and let $S_1, S_2 \subset X$ be two smooth surfaces such that the scheme theoretic
intersection $S_1 \cap S_2$ contains a smooth curve $C$ of degree $d$ and genus $g$. Then:

\begin{equation}
    mult_C(S_1,S_2) = 3d +K_{S_1}\cdot C + K_{S_2}\cdot C +2 -2g,
\end{equation}

where $K_{S_i}$ denotes the canonical class of $S_i$ and $mult_C(S_1 \cdot S_2)$ the multiplicity of
intersection of $S_1$ and $S_2$ along $C$.
\end{prop}

\newcommand{\GG}{\mathbb{G}}

\section{Cubic hypersurfaces in $\mathcal C_{12}$}\label{c12}

One can define the divisor $\mathcal C_{12}\subset\mathcal C$ in two possible ways. The first one is as the locus of smooth cubic hypersurfaces having a marking of discriminant 12, the second is as the closure of the locus of cubic hypersurfaces containing
a smooth cubic rational normal scroll $S$. 
More precisely, we say that a cubic fourfold has a marking of discriminant 12 whenever it contains a surface, whose class $S$ verifies $S\cdot S=7$ and $S\cdot h^2=3$. A smooth cubic rational normal scroll has some possible degenerations. First, it can degenerate to a cone over a twisted cubic; otherwise, it could deform to the union of a quadric surface $Q$ and a plane $P$ such that $Q\cap P=\langle Q \rangle \cap P$ is a line $L$ and $\langle Q \rangle$ is the linear envelope of the quadric surface. The second case can further degenerate to the union of three planes if $Q$ degenerates further to the union of two planes. These three degenerations are small varieties, and are contained
in smooth cubic hypersurfaces. There are other degenerations obtained by projection but the resulting
schemes are not contained in any smooth cubic hypersurface in $\PP^5$. Hence every element
$X\in \mathcal C_{12}$ contains either a cubic rational normal scroll or a reducible surface $Q\cup P$ as above, where $Q$ is a quadric
surface, irreducible or not. If $Q\cup P\subset X$, then $X\in \mathcal
C_{12}\cap\mathcal C_8$ and $P\cdot (Q+P)=3$. Furthermore, we observe that if $S=P\cup Q$, letting $P_1+Q=h^2$, we see that there exists a plane $P_1\in<h^2, Q, P>$
such that $P_1\cdot P=1$ and $P_1\cdot (Q+P)=-1$. 
\medskip

\begin{thm}\label{cap812} There are five irreducible components of $\mathcal C_8\cap \mathcal C_{12}$ 
	indexed by  the value $P\cdot S = \epsilon \in\{-1, 0, 1, 2, 3\}$, where $P$ is a plane and $S$ the class
	of a cubic rational normal scroll. For $\epsilon=0$ or $2$ every element in the corresponding irreducible component
	is rational. 
\end{thm}

\begin{proof}

First of all we observe that the generic cubic $X$ contained in the intersection $\mathcal C_8\cap\mathcal C_{12}$ has a sublattice $\langle h^2, P, S \rangle \subset M(X)$. The intersection lattice of $\langle h^2, P, S \rangle$ is as follows,

\begin{equation}\label{eq:disc812}
\begin{tabular}{|c|c|c|c|}
\hline     &  $h^2$ & $P$ & $S$\\
\hline  $h^2$ & 3 & 1  & 3 \\
\hline   $P$ & 1 & 3 & $\eta$ \\
\hline  $S$ & 3 & $\eta$ & 7 \\
\hline
\end{tabular},
\end{equation}
for some $\eta=P\cdot S \in \mathbb{Z}$. Let us denote by $M_{\eta}$ the rank 3 sublattice.\\
Note that $M(X)$ is definite positive by the Hodge-Riemann bilinear relations. Hence, by Sylvester criterion, $M_{\eta}$ must have positive determinant. The only integer values of $\eta$ for which its determinant $-3\eta^2+6\eta + 29$ is positive are $\pm 2, \pm 1, 0,3,4$. We will now show that there are no such cubics for $\eta = -2, 4$. By \cite[Lemma 2.4]{YYu}, the existence of a short root (\it i.e. \rm an integer vector with norm 2) is equivalent to the emptiness of the component. If $\eta=-2$, we observe that 
$\left(\begin{array}{c} -2 \\ 1 \\ 1\
\end{array}\right)$ is a short root, and if $\eta=4$ then $\left(\begin{array}{c} 0 \\ -1 \\ 1\end{array}\right)$ has also norm 2.

\smallskip

An easy explicit calculation shows that no other component admits short roots, hence by \cite[Lemma 2.4]{YYu}, for all the other values of $\eta$ the components are non-empty. Let us denote by $\mathcal{C}_{M_{\eta}}$ the locus of cubic fourfolds such that there is a primitive embedding of $M_{\eta} \subset M(X)$ preserving $h^2$.

In order to check irreducibility we need to check that there are no overlattices $B$ of $M_\eta$, that embed primitively in 
$H^4(X,\mathbb{Z})$. If there exists an embedding $M_\eta \hookrightarrow B$, then $|d(M_\eta)|=|d(B)|\cdot[B:M_\eta]^2$. Moreover, depending on the value of $\eta$, we have

\begin{equation}
\begin{tabular}{|c|c|}
\hline  $\epsilon$   &  $d(M_\eta)$\\
\hline -1& 20\\
\hline    0 & 29\\
\hline  1&  32\\
\hline 2 &  29\\
\hline   3  &20\\
\hline
\end{tabular}.
\end{equation}

We observe straight away that 29 is square free, hence $M_2$ and $M_0$ cannot have overlattices. For the other components, we will show that any possible overlattice has short roots, so that there are no smooth cubic fourfolds with such a lattice.

Consider $h^2$ and $P$ as a part of a basis of the overlattice
$B$ and let $U$ be a vector that completes this to a basis such that $U = xh^2+yP +zS$,
with $x, y, z$ rational coefficients. Consider now the natural embedding of $M_\eta$ in $B$. This can be written explicitly as follows:

$$\left(\begin{array}{ccc}
1  & 0 & x  \\
0 & 1 & y \\
0 & 0 & Z \\
\end{array}\right)^{-1} = \left(
\begin{array}{ccc}
    1 & 0 & \frac{-x}{z} \\
    0 & 1 & \frac{-y}{z} \\
    0 & 0 & \frac{1}{z}
\end{array}\right) \in \mathcal{M}_{3,3}(\mathbb Z).$$

If we take $z = \frac{1}{n}$, for some $n \in \mathbb Z$ and $x' = nx$, $y' = yn \in \mathbb{Z}$, then we can write $U = \frac{1}{n}(x'h^2 +
y'P + S)$. Moreover, by adding appropriately multiples of $h^2$ and of $P$, we can assume that $0 \leq x', y' < n$.
An explicit computation of the intersections gives the the following:

\begin{align*}
U.h^2 & =  \frac{1}{n}
(3x' + y' + 3) = a, \\
U.P & =  \frac{1}{n} (x' + 3y' + \eta ) = b, \\
U.U & =  \frac{1}{n^2} (3x'^2 + 3y'^2 + 6x' + 2\eta y' + 2x'y' + 7) = c. 
\end{align*}

With this notation, the Gram matrix of $B$ is 

$$\left(
\begin{array}{ccc}
    3 & 1 & a \\
    1 & 3 & b \\
    a & b & c
\end{array}\right).$$

Since the discriminant of $M_{-1}$ is 20,   $[B:M_{-1}]$ can only be 4 and $n=2$. The only cases where the Gram matrix of B has integer entries are:
 
 \begin{itemize}
     \item $n=2,\ x'=1,\ y'=0,$ which gives $a=3,\ b=0,\ c=4;$
     
     \item $n=2,\ x'=0,\ y'=1,$ which gives $a=2,\ b=1,\ c=2.$
     
 \end{itemize}
 In the first of the two cases, one computes explicitly that the Gram matrix of B has $\left(\begin{array}{c} -1 \\ 1 \\ 1 
\end{array}\right)$ as a short root, hence it is not a lattice of a cubic fourfold. In the second case, the Gram matrix has a short root $\left(\begin{array}{c} 0 \\ 0 \\ 1 
\end{array}\right)$, hence the component $\mathcal{C}_{M_{-1}}$ is irreducible.
 
\medskip Let us now consider $M_1$. Since the discriminant of $M_1$ is 32, then $[B:M_1]^2$ is either 4 or 16, and $n=2$ or $4$. Hence, the only cases where the matrix above has integer entries are the following two:

\begin{itemize}
    \item $n=2,\ x'=1,\ y'=0,$ which gives $a=3,\ b=1,\ c=4;$\\
    \item $n=2,\ x'=0,\ y'=1,$ which gives $a=2,\ b=2
    ,\ c=3.$
\end{itemize}

In the first case the corresponding Gram matrix has no short roots. On the other hand one can check easily that the two vectors $(-1,3,0)$ and $(0,3,-1)$ make up a basis for the rank 2, primitive lattice $B_{prim}:= (h^{2})^{\perp}\subset H_{prim}^4(X,\mathbb{Z})$. The Gram matrix of $B_{prim}$ is $\left(\begin{array}{cc} 24 &  24 \\ 24 & 25 \\ 
\end{array}\right)$, hence the lattice is not even. Since $H_{prim}^{4}(X,\mathbb{Z})$ is even, $B$ cannot be an overlattice of $M_1$. 

\smallskip

In the second case, the Gram matrix has a short root, that is $\left(\begin{array}{c} -1 \\ 0 \\ 1 
\end{array}\right)$. The value $n=4$ does not give a integer valued matrix. This means that $\mathcal C_{M_1}$ is irreducible as well.

\smallskip

For the last case that is left, since the discriminant of $M_{3}$ is 20, the index $[B:M_{3}]$ can only be 4 and $n=2$. The only cases where the Gram matrix of B has integer entries are:
 
 \begin{itemize}
     \item $n=2,\ x'=0,\ y'=1,$ which gives $a=2,\ b=3,\ c=4;$
     \item $n=2,\ x'=1,\ y'=0,$ which gives $a=3,\ b=2,\ c=4;$
     
 \end{itemize}
In both cases, the Gram matrix of B has a short root $\left(\begin{array}{c} -1 \\ -1 \\ 1 
\end{array}\right)$. Hence $\mathcal{C}_{M_3}$ is irreducible.

\medskip

\it Rationality of $X\in \mathcal{C}_{M_0}:$ \rm It is not hard to see that the general element of the family where $ \#(P\cap S)= 0$ is also rational. In fact we have  $\rk(A(X))=3$ and let us set $Q+P=h^2$. Since $S\cap P=\emptyset$, we have  $S\cdot Q=3$ and
we deduce that the rational quadric fibration over $\PP^2$ determined by projection from $P$ has a rational section, hence it is rational by Prop. \ref{multisection}. Since the general cubic in this component is rational, by \cite{KT} this holds for all cubics with $\#(P\cap S)= 0$.

\medskip

\it Rationality of $X\in \mathcal{C}_{M_2}:$ \rm Every cubic fourfold $X\in \mathcal C_{M_2}$ is rational. Let $p\in \PP^5$ be a general point, and the projections of $S$ and $P$ off $p$ intersect in 3 points, but 2 of them come from $S\cap P$ in $\PP^5$. This means that there is only one secant line through $p$, that is $S\cup P\subset \PP^5$ is a reducible OADP and $X$ is then rational. We also observe that $P\not\subset<S>=\PP^4$ since $P\cdot S=2$ implies that $P$ and $S$ cut only at two points.

\end{proof} 
 
\begin{lema}\label{12coincide}
The irreducible components $\mathcal{C}_{M_{-1}}$ and $\mathcal{C}_{M_{3}}$ coincide, the same holds for $\mathcal{C}_{M_{0}}$ and $\mathcal{C}_{M_{2}}$.
\end{lema}

\begin{proof}
By applying \cite[Algorithm 7.6]{yangyu}, we find that the generic element of any irreducible component of the intersection $\mathcal C_{12}\cap \mathcal C_{8}$ has an intersection lattice of type

$$\left(
\begin{array}{ccc}
    3 & 1 & 0 \\
    1 & 3 & \tau \\
    0 & \tau & 4
\end{array}\right),$$

for certain choice of the basis of the lattice and $\tau\in \mathbb{Z}$. Moreover, lattices obtained from $\tau$ and $-\tau$ are isometric. The determinant of this matrix is $32-3\tau^2$ 
 hence it takes the value 20 only for $\tau=\pm 2$, and hence there is a unique irreducible component with discriminant 20 and $\mathcal{C}_{M_{-1}}$ and $\mathcal{C}_{M_3}$ coincide. 

\smallskip

The same argument holds for $\mathcal{C}_{M_0}$ and $\mathcal{C}_{M_2}$, with values of $\tau=\pm 1$.
\end{proof}

\bigskip

\subsection{Explicit examples of cubic fourfolds in the irreducible components of $\mathcal{C}_{12}\cap\mathcal{C}_8$}\label{compC12S}{\rm ($ \mathcal{C}_{M_\epsilon}\neq\emptyset$ for $\epsilon\in\{-1,0,1,2,3\}$)\\

First of all, we need two technical results.

\begin{prop}\label{oadpvsdef}
Let $Y=\tilde{S}\cup \tilde{P}\subset \PP^5$ be a reducible quartic surface given by the union of a plane $\tilde{P}$ and a smooth cubic scroll $\tilde{S}$ intersecting along a line $\tilde{L}=\tilde{S}\cap \tilde{P}$. Then $Y$ is OADP if and only if $\tilde{L}^2=0$ (that is, $\tilde{L}$ is a line of the ruling of $\tilde{S}$) and secant defective if and only if $\tilde{L}^2=-1$ (that is, $\tilde{L}$ is the directrix line of $\tilde{S}$) .
\end{prop}

\begin{proof}
Consider a cubic scroll $S\subset \PP^4$ and a plane $P$ also in $\PP^4$ that intersects $S$ along a line $L.$ Let $K_S$ be the canonical sheaf. We observe that $K_S\cdot L = -2 - L^2$. Then, by \cite[Prop. 9.1.1, third formula]{Fu} , we have that the multiplicity of the scheme theoretic intersection $S\cap P$ along $L$ 
is given by

\begin{align*} 
mult_L(S \cap P) &= 5 + K_S\cdot L + K_P\cdot L +2 \\
& = 5 - 2 - L^2 -3 +2 \\
& = 2 - L^2.
\end{align*}

This formula gives 2 or 3 depending on whether $L^2=0$ or $L^2=-1$.
Suppose now that we have our $Y\subset \PP^5$, then the computation above implies that $Y$ is OADP if and only if $\tilde{L}^2=0$. In fact, let us consider a general $p\in\PP^5$ and call $S$ and $P$ the images of $\tilde{S}$ and $\tilde{P}$ in $\PP^4$ via the projection off $p$. If $S$ and $P$ intersect along a set $\Omega$, then the number of secants to $Y$ through $p$ is equal to $deg(S)\cdot deg(P)- mult_{\Omega}(S\cap P)$, which gives $3 - mult_L(S \cap P)$ in our particular case. This implies that $Y$ is OADP if $mult_L(S\cap P)=2$, and secant defective if $mult_L(S\cap P)=3$. 
\end{proof}

\begin{prop}\label{verovsscroll}
Let $Y=\tilde{S}\cup \tilde{P}$ be as in Prop. \ref{oadpvsdef}. Then $Y$ is OADP if and only if it is the flat limit of a family of quartic scrolls in $\PP^5$; $Y$ is secant defective if and only if it is the flat limit of Veronese surfaces in $\PP^5$.
\end{prop}

\begin{proof}

By checking the Hilbert polynomial, one  observes that the reducible surface $\tilde{S}\cup \tilde{P}$ can be a degeneration of either a Veronese surface, or a quartic scroll. By using a construction contained in \cite[Sect. 3]{CMR}, we will show that either cases can occur, and that, in particular, it is the self intersection of the line $\tilde{S}\cap \tilde{P}$ that determines in which case we are.

\smallskip

 Let $\Sigma$ be a Veronese surface in $\PP^5$. Let us consider the family of morphisms $\pi_p: \Sigma \to \PP^4$ given by the linear projection off a point $p\in\PP^5$, generically not lying on $\Sigma$ but whose limit $\hat{p}$ lies on the surface. The limit map, i.e. the projection off $\hat{p}\in \Sigma$, is not a morphism. Nevertheless, following \cite[Sect. 3]{CMR}, we can resolve it as a morphism $\tilde{\pi}_{\hat{p}}$ by replacing $\Sigma$ with $\widehat{\Sigma}\cup Z$, where $\widehat{\Sigma}$ is the blow-up of $\Sigma$ in $\hat{p}$ and $Z$ a $\PP^2$ intersecting $\widehat{\Sigma}$ along the exceptional divisor. In particular, $\widehat{\Sigma}\cup Z$ is the flat limit of a family of surfaces where each other element is a Veronese surface. Since $\Sigma$ does not contain lines, we observe that  $\widehat{\Sigma}\cup Z$ is sent isomorphically by $\tilde{\pi}_{\hat{p}}$ onto a reducible surface in $\PP^4$ whose components are a cubic scroll $S$ plus a plane $B$ (in fact the degree is 4) that intersects the scroll along the exceptional divisor of $\hat{p}$, which is a line. The fact that $S$ is the cubic scroll is due to the fact that it is exactly the image of the blow-up of $\widehat{\Sigma}$. Moreover, the lines inside $S$ are the projections of the conics inside $\Sigma$ passing through $\hat{p}$, and they all intersect the exceptional line. This means that the exceptional divisor is the $(-1)$-curve of $S$, and by Prop. \ref{oadpvsdef} this is equivalent to $Y$ being secant defective.

 On the other hand, if $\Sigma$ is a quartic scroll, the argument of \cite{CMR} runs basically the same way. The only difference is that the morphism $\tilde{\pi}_{\hat{p}}$ that resolves the projection off $\hat{p}$ contracts a line, giving the right morphism from the (blown-up) quartic scroll to the cubic scroll. The images of the lines of the ruling of $\Sigma$ do not intersect the exceptional line, that just replaces the contracted fiber, and hence has zero self-intersection. By Prop. \ref{oadpvsdef}, this is equivalent to $Y$ being OADP.

\end{proof}

\begin{Remark}
This degeneration of the Veronese surface has been considered also in \cite[Thm. 2.5]{GR}.
\end{Remark}

\subsubsection{$S\cap P$ is a finite set.} In the appendix, via an easy Macaulay calculation \cite{macaulay2}, we display examples of smooth
cubic hypersurfaces $X\subset\PP^5$ containing a smooth rational normal scroll  $S$ of degree 3 and a plane $P$ such that $S\cap P$ is a scheme
of length $\epsilon$, $\epsilon\in\{0,\ldots,3\}$ consisting exactly of $\epsilon$ reduced points, see Example \ref{exscroll} $(c,d,e,f)$.
A general element $X\in\mathcal {C}_{M_\epsilon}$, $\epsilon\in\{0,1,2,3\}$,   has $\rk(M(X))=3$.

\subsubsection{$S\cap P$ is a conic.} In Example \ref{exscroll} $(b)$ , we construct $X\subset\PP^5$ a smooth cubic hypersurface containing a smooth rational normal scroll $S$ of degree 3 and a plane $P$ spanned by a conic $C\subset S$. By Eq. \eqref{eq:excess}, we have $S\cdot P=-1$, hence $X\in \mathcal{C}_{M_{-1}}$. In this case $S\cup P\subset <S>=\p^4$ is a degenerate reducible surface whose hyperplane sections have arithmetic genus one, since they are given by a twisted cubic plus one of its secant lines.

\subsubsection{$S\cap P$ is a line of the ruling.} Let $X\subset\PP^5$ be a smooth cubic hypersurface containing a smooth rational normal scroll $S$ of degree 3 and a plane $P$ such that $P\cap S=P\cap <S>$ is a line of the ruling of $S$. An example of such a cubic fourfold is developed in Ex. \ref{exscroll} $(7)$.   By \eqref{eq:excess}, in this case we have  $S\cdot P=0$, so that $X\in \mathcal{C}_{ M_0}$. In particular  such cubic hypersurfaces are rational since $S\cup P$ is a (small, equidimensional) OADP surface (see Prop. \ref{oadpvsdef} for an explanation).

More precisely, $S\cup P$ is a flat limit of a family of quartic scrolls in $\PP^5$ (see Prop. \ref{verovsscroll}), hence the corresponding cubics fourfold belong to $\mathcal{C}_8\cap \mathcal{C}_{12}\cap \mathcal{C}_{14}$. These examples of cubic fourfolds hence lie also in the irreducible component of $\mathcal{C}_8\cap \mathcal{C}_{14}$ where the plane intersects the scroll in 3 points (see \cite{ABBVA,BRS}), in fact $P\cdot(S\cup P) = 0+3$. Indeed, all cubics in this component have Gram matrix of discriminant 29, which is the same as the discriminant of those in $\mathcal{C}_{M_0}\subset \mathcal{C}_8\cap \mathcal{C}_{12}$.

\subsubsection{$S\cap P$ is the directrix of $S$.}Let $X\subset\PP^5$ be a smooth cubic hypersurface containing a smooth rational normal scroll $S$ of degree 3 and a plane $P$ such that $P\cap S=P\cap <S>$ is
the directrix line of $S$. We produce such examples in the Appendix (Ex. \ref{exscroll} (8)). In this case we have $S\cdot P=1$ by \eqref{eq:excess}, hence  $X\in\mathcal{C}_{M_1}$. The scheme $S\cup P\subset 
\PP^5$ is a reducible, secant defective surface (see Prop. \ref{oadpvsdef}) which is a flat projective degeneration of a Veronese surface in $\PP^5$ (see Prop. \ref{verovsscroll}) . In particular the secant lines to $S\cup P$ fill two irreducible varieties of dimension 4: $<S>$ and the join of $S$ and $P$ which is a quadric surface with vertex $P$. As a consequence, these cubics are contained in $\mathcal{C}_{20}\cap \mathcal{C}_8\cap \mathcal{C}_{12}$ and notably in the component of $\mathcal{C}_{20}\cap \mathcal{C}_8$ where $P \cdot V= P\cdot (S\cap P)= 1+3=4$ (see Sect. \ref{tripleint} for a full description of this component).

\section{Cubic hypersurfaces in  $\mathcal C_{20}$}\label{c20}

The closure of the locus of smooth cubic hypersurfaces $X\subset\p^5$ containing a Veronese surface
$V\subset\p^5$ is indicated with $\mathcal C_{20}$. Since $V^2=12$, the discriminant of a general
$X\in\mathcal C_{20}$ with $\rk(M(X))=2$ is exactly 20.

\begin{thm}\label{compC208} There are seven  irreducible components of $\mathcal C_8\cap \mathcal C_{20}$
 indexed by  $P\cdot V=\gamma\in\{-2,-1, 0, 1, 2, 3, 4\}$, where $P$ 
 is a plane and $V$ the class of
a surface such that $V^2=12$ and $V\cdot h^2=4$. For $\gamma=-1,1$ and $3$ each smooth cubic hypersurface
belonging to the corresponding irreducible component is rational.
\end{thm}

\begin{proof} Let us first observe that the generic cubic $X$ contained in the intersection $\mathcal C_8\cap\mathcal C_{20}$ has intersection lattice as follows,

\begin{equation}\label{eq:disc812}
\begin{tabular}{|c|c|c|c|}
\hline     &  $h^2$ & P & V\\
\hline  $h^2$ & 3 & 1  & 4 \\
\hline   P & 1 & 3 & $\gamma$ \\
\hline  V & 4 & $\gamma$ & 12 \\
\hline
\end{tabular},
\end{equation}

where $V$ is the class of a Veronese surface inside the cubic fourfold $X$. Let us denote by $N_\gamma$ the lattice with this intersection matrix. The determinant of this matrix is $-3\gamma^2 + 8\gamma +48$, and the only integer values for which it is positive are $\{-2,-1,0,1,2,3,4,5\}$. These values could possibly give non-empty components of the intersection $\mathcal C_8\cap\mathcal C_{20}$. We will denote by $\mathcal{D}_{N_\gamma}$ the locus of cubic fourfolds $X$ such that $N_\gamma\subset M(X)$ is a saturated sublattice.

For $\gamma=5$, it is straightforward to see that $\left(\begin{array}{c} 1 \\ 1 \\ -1\
\end{array}\right)$ is a short root, hence $\mathcal{D}_{N_5}$ is empty. 

The other possible values of $\gamma$ give rise to lattices with no short roots, hence by \cite[Lemma 2.4]{YYu} the associated components $\mathcal{D}_{N_\gamma}$ are non-empty. Hence we are left only with the values $-2,-1,0,1,2,3,4$. The corresponding discriminants are:

\begin{equation}\label{discr820}
\begin{tabular}{|c|c|}
\hline  $\gamma$   &  $d(N_\gamma)$\\
\hline -2 & 20\\
\hline -1& 37\\
\hline    0 & 48\\
\hline  1&  53\\
\hline 2 &  52\\
\hline   3  &45\\
\hline   4  & 32\\
\hline
\end{tabular}
\end{equation}

\medskip

In order to check irreducibility we need to check that there are no overlattices $B$ of $N_\gamma$, that embed primitively in 
$H^4(X,\mathbb{Z})$. If there exists an embedding $N_\gamma \hookrightarrow B$, then $|d(N_\gamma)|=|d(B)|\cdot[B:N_\gamma]^2$.
For $\gamma= -1, 1$, the discriminant are square-free, hence $N_1$ and $N_{-1}$ cannot have proper finite overlattices, and they correspond to irreducible components.

\medskip

We need some elementary lattice theory to prove the irreducibility of the components with $\gamma=-2,\ 0,\  2,\ 3,\ 4$.

 Consider $h^2$ and $P$ as a part of a basis of the overlattice $B$ and let $U$ be a vector that completes this to a basis such that $U = xh^2+yP +zV$,
with $x, y, z$ rational coefficients. Consider now the natural embedding of $N_\gamma$ in $B$. This can be written explicitly as follows:

$$\left(\begin{array}{ccc}
1  & 0 & x  \\
0 & 1 & y \\
0 & 0 & Z \\
\end{array}\right)^{-1} = \left(
\begin{array}{ccc}
    1 & 0 & \frac{-x}{z} \\
    0 & 1 & \frac{-y}{z} \\
    0 & 0 & \frac{1}{z}
\end{array}\right) \in \mathcal{M}_{3,3}(\mathbb Z).$$

If we take $z = \frac{1}{n}$, for some $n \in \mathbb Z$ and $x' = nx$, $y' = yn \in \mathbb{Z}$, then we can write $U = \frac{1}{n}(x'h^2 +
y'P + S)$. Moreover, by adding appropriately multiples of $h^2$ and of $P$, we can assume that $0 \leq x', y' < n$.
An explicit computation of the intersections gives the the following:

\begin{align*}
U.h^2 & =  \frac{1}{n}
(3x' + y' + 3) = a, \\
U.P & =  \frac{1}{n} (x' + 3y' + \eta ) = b, \\
U.U & =  \frac{1}{n^2} (3x'^2 + 3y'^2 + 8x' + 2\eta y' + 2x'y' + 12) = c. 
\end{align*}

With this notation, the Gram matrix of $B$ is 

$$\left(
\begin{array}{ccc}
    3 & 1 & a \\
    1 & 3 & b \\
    a & b & c
\end{array}\right).$$

Since the discriminant of $N_{-2}$ is 20, hence $[B:N_{-2}]$ can only be 4 and $n=2$. The only cases where the Gram matrix of B has integer entries are:
 
 \begin{itemize}
     \item $n=2,\ x'=0,\ y'=0,$ which gives $a=2,\ b=-1,\ c=3;$
     
     \item $n=2,\ x'=1,\ y'=1,$ which gives $a=4,\ b=1,\ c=6.$
     
 \end{itemize}
 In the first of the two cases, one computes explicitly that the Gram matrix of B has $\left(\begin{array}{c} -1 \\ 0 \\ 1 
\end{array}\right)$ as a short root, hence it is not a lattice of a cubic fourfold. In the second case, the Gram matrix has a short root $\left(\begin{array}{c} -2 \\ 0 \\ 1 
\end{array}\right)$, hence the component $\mathcal{D}_{N_{-2}}$ is irreducible.\\

Consider now $N_0$. Since the discriminant of $N_{0}$ is 48, hence $[B:N_0]$ can only be 16 (resp. 4) and $n=4$ (resp. $n=2$). The only cases where the Gram matrix of B has integer entries are:
 
 \begin{itemize}
     \item $n=2,\ x'=0,\ y'=0,$ which gives $a=2,\ b=0,\ c=3;$
     
     \item $n=2,\ x'=1,\ y'=1,$ which gives $a=4,\ b=1
     2,\ c=7.$
     
 \end{itemize}
 In both cases, one computes explicitly that the Gram matrices of B have $\left(\begin{array}{c} -1 \\ 0 \\ 1 
\end{array}\right)$ as a short root, hence they are not lattices of a cubic fourfold. Hence the component $\mathcal{D}_{N_{0}}$ is irreducible.\\

Let us consider now $N_2$. Since the discriminant of $N_{2}$ is 52, hence $[B:N_{2}]$ can only be 4 and $n=2$. The only cases where the Gram matrix of B has integer entries are:
 
 \begin{itemize}
     \item $n=2,\ x'=0,\ y'=0,$ which gives $a=2,\ b=1,\ c=3;$
     
     \item $n=2,\ x'=1,\ y'=1,$ which gives $a=4,\ b=3
    ,\ c=8.$
     
 \end{itemize}
 In the first case, the Gram matrix of B has $\left(\begin{array}{c} -1 \\ 0 \\ 1 
\end{array}\right)$ as a short root, hence it is not the lattice of a cubic fourfold.\\ In the second case, the Gram matrix of B has $\left(\begin{array}{c} -1 \\ -1 \\ 1 
\end{array}\right)$ as a short root. Hence the component $\mathcal{D}_{N_{2}}$ is irreducible.\\

The discriminant of $N_3$ is 45, hence $[B:N_{3}]$ can only be 9 and $n=3$. The only case where the Gram matrix of B has integer entries is:
 
 \begin{itemize}
     \item $n=3,\ x'=0,\ y'=2,$ which gives $a=2,\ b=3,\ c=4;$     
 \end{itemize}
which has $\left(\begin{array}{c} -1 \\ -1 \\ 1 
\end{array}\right)$ as a short root, hence it is not the lattice of a cubic fourfold and the component $\mathcal{D}_{N_{3}}$ is irreducible.

Last, the discriminant of $N_4$ is 35, hence $[B:N_{4}]$ can be 16 (resp. 4) and $n=4$ (resp. 2). The only cases where the Gram matrix of B has integer entries are:
 
 \begin{itemize}
     \item $n=2,\ x'=0,\ y'=0,$ which gives $a=2,\ b=2,\ c=3;$
          \item $n=2,\ x'=1,\ y'=1,$ which gives $a=4,\ b=4,\ c=9;$     
 \end{itemize}
which has $\left(\begin{array}{c} -1 \\ 0 \\ 1 
\end{array}\right)$ as a short root for the first case. Hence it is not a lattice of a cubic fourfold.\\
For the second case, Let us set $a=(1,-3,0)$ and $b=(0,-4,1)$ the basis of $B_{prim}$. Its Gram matrix is $\left(\begin{array}{cc} 24 &  24 \\ 24 & 25 \\ 
\end{array}\right)$ which is not even. This shows that $\mathcal{D}_{N_4}$ is an irreducible component.

\medskip

\it Rationality of $X\in \mathcal{D}_{N_{-1}}:$
\rm If $\gamma=-1$, then $X$ is rational  since all cycles $C=\alpha h^2 + \beta P + \delta V \in M(X)$, with $\alpha,\beta \in \mathbb{Z}$ and $\delta$ an odd integer, intersect the quadric surfaces of class $Q=h^2-P$ in an odd degree zero cycle. This implies rationality of $X$ by Prop. \ref{multisection}.

\it Rationality of $X\in \mathcal{D}_{N_{3}}:$
\rm If $\gamma=3$, then each $X$ in the corresponding irreducible component is rational since it contains a Veronese surface $V$ and a plane $P$ intersecting $V$ in three points. The union $V\cup P$ is a reducible OADP surface. One can see this fact by remarking that the images of $V$ and $P$ in $\PP^4$, via the projection off a generic point $p\in \PP^5$, intersect in 4 points, but 3 of them come from $V\cap P$ in $\PP^5$. This means that only one secant line to $V\cap P$ passes through $p$ in $\PP^5$.

\it Rationality of $X\in \mathcal{D}_{N_{1}}:$
\rm If $\gamma=1$, then a cubic hypersurface  in the corresponding irreducible component  is rational. In fact,
letting $Q=h^2-P$ as before, we have   $Q\cdot V=4-\gamma$ so that the rational quadric fibration defined by projection from $P$
has a rational section and the total space is hence rational by \cite[Cor. 2.2]{Ha}.

\end{proof}

\subsection{Explicit examples of cubic fourfolds in the irreducible components of $\mathcal{C}_{20} \cap \mathcal{C}_8$.}

In this section we will describe the explicit examples of smooth cubic fourfolds $X\subset \PP^5$ containing a smooth non-degenerate Veronese surface $V\subset \PP^5$ and a plane $P$, such that $V\cdot P =\gamma$. The first crucial observation (that reflects also our lattice theoretical computation) is that there are no such cubics if $\gamma<-1$ or if $\gamma>3$, as we will see here below in \ref{VPconic} and \ref{tripleint}. There are cubics that contain surfaces with this intersection theoretical properties, but the Veronese then is not smooth.

\subsubsection{$V\cap P$ is a conic.}\label{VPconic} Assume $V$ is a smooth Veronese surface. If $\gamma<0$, then the cycle $P\cdot V$ must contain a curve. Since $V$ is defined by quadratic equations, the scheme-theoretic intersection $P\cap V$ must be a conic $C\subset V$ (recall that a Veronese surface contains no lines).
From the fact that $K_V\cdot C=-3$, we easily deduce $V\cdot P=-1$ by \eqref{eq:excess}. An example of such a cubic is given in Example \ref{exveronese}  (1). 
Moreover, this argument also implies that for cubics in the components indexed by $\gamma=-2$, the class of the Veronese will represent singular/reducible degenerations of the surface, since otherwise $V \cdot P=-1$.

\begin{prop}
The union $W:=V\cup P\subset \PP^5$ of a Veronese surface and a projective plane, that intersect each other along a conic $C$, is a reducible OADP surface.
\end{prop}

\begin{proof}
Let us denote as usual by $K_V$ and $K_P$ the two canonical sheaves. First of all we observe that by Equation \ref{eq:excess} we have $V\cdot P= 2 +K_V\cdot C=-1$. Let us project $W$ off a generic point $p\in \PP^5$. By the same argument as in Prop. \ref{oadpvsdef}, we compute the multiplicity in $C$ of the intersection of the two surfaces in $\PP^4$. We abuse of notation by sticking to the same notation as in $\PP^5$. By \cite[Prop. 9.1.1, third formula]{Fu}, then we obtain

\begin{align*}
mult_C(V \cap P) &= 10 + K_V\cdot C + K_P\cdot C +2 \\
& = 10 - 6 - 3 + 2 \\
& = 3.
\end{align*}

The number of secants to $W$ through $p$ is equal to $deg(V)\cdot deg(P) - mult_C(V \cap P)= 4-3 =1.$ This completes the proof.

\end{proof}

\begin{prop}
Let $W=V\cup P\subset \PP^5$ the union of a Veronese surface and a projective plane, that intersect along a conic. Then $W$ is a flat deformation of a del Pezzo quintic surface.
\end{prop}

\begin{proof}
Let us consider the rational map $\varphi: \PP^3 \dashrightarrow \PP^6$ given by quadrics through 3 generic base points. The image is a quintic del Pezzo threefold $T$ and the image of a smooth quadric through the 3 points is a quintic del Pezzo surface. Consider the degenerate quadric $Q$ given by the union of a generic plane and the plane through the 3 fixed points. It is straightforward to see that the image of $Q$ is exactly $V\cup P$ and $V\cap P$ is a conic. This concludes the proof.
\end{proof}

Moreover, this argument also implies that for cubics in the components indexed by $\gamma=-2$ and $4$, the class of the Veronese will represent singular/reducible  degenerations (for the case $\gamma=4$ see Sect. \ref{tripleint} and Prop. \ref{oadpvsdef}).

\subsubsection{$V\cap P$ is a finite set, and $\#(V\cap P) \leq 3$}\label{VPfinite} If $\gamma\geq 0$ and we can suppose that $P\cap V$ is finite and that it consists of $\gamma$ points counted with multiplicity.
Since $V$ is scheme theoretically defined by quadratic equations whose first syzygies are generated by the linear ones, we get $0\leq \gamma\leq 3$ by the last part of \cite[Thm. 1.1]{EGHP}. Explicit equations for these kind of cubics are developed in Ex. \ref{exveronese} (2,3,4,5). 


\subsubsection{$V\cdot P=4$ and $V$ is singular}\label{tripleint}
The generic element of the component where $V\cap P=4$ can not be a smooth Veronese surface. In fact, 4 points on $V$, that lie all on the same plane, come from 4 points on $\PP^2$ that lie on 3 conics, which is not possible. Hence the generic element of this component should be singular or reducible.

\smallskip

From Sect. \ref{compC12S}, we recall that we constructed examples of cubics in $\mathcal{C}_8 \cap \mathcal{C}_{12}$ where $S \cdot P=1$ and $S\cap P$ is the directrix line of the cubic scroll. By Prop. \ref{verovsscroll}, these quartic surfaces are flat limits of families of Veronese surfaces. An easy Hilbert polynomial computation shows that whenever the union of a cubic scroll and a plane is a flat degeneration of a Veronese surface, the intersection $S\cap P$ must be a line. Moreover the Gram matrices of cubic fourfolds in $\mathcal{C}_8 \cap \mathcal{C}_{12}$ with $S \cdot P=1$ have discriminant 32, exactly like cubics in $\mathcal{C}_{20}\cap \mathcal{C}_8$ with $V\cdot P=4$. In fact, in this case, the self intersection of $S\cup P$ is 12, and $P\cdot(S\cup P)= 1+3=4$, and we recognize a "reducible Veronese", and the cubics containing such a configuration of surfaces live inside $\mathcal{D}_{N_4}$. This means that the locus of cubics containing $S$ and $P$ that intersect along the directrix make up a locus that is contained in the intersection $\mathcal{C}_8\cap \mathcal{C}_{20} \cap \mathcal{C}_{12}$, which is probably not cut out by other Hassett divisors, and is contained in $\mathcal{D}_{N_4}\subset \mathcal{C}_{20}$.

\medskip

\newpage

\appendix
\section{Examples of smooth cubic hypersurfaces in $\mathbb{P}^5$ and two Hassett divisors. \\   }
We shall construct explicitly, by using the software system 
Macaulay2 \cite{macaulay2},
smooth cubic hypersurfaces in $\mathbb{P}^5$ which contain 
a given smooth rational  surface 
and a plane in a certain position with regard to the surface.
More precisely, 
we shall exhibit a  smooth cubic hypersurface of $\mathbb{P}^5$ containing a Veronese surface 
 (resp. a cubic rational normal scroll) 
 and a plane intersecting the surface along one of the following schemes:
  an irreducible conic; a line; 
a set of $1\leq i\leq 3$ linearly independent reduced points;
  the empty scheme. We work over the finite field $\mathbb{F}_{31}$ but our equations hold
over fields of characteristic zero. Hence we will set $\mathbb{P}^5 := Proj(\mathbb{F}_{31}[x_0,\dots, x_5])$ and
$\mathbb{P}^2 := Proj(\mathbb{F}_{31}[t_0,\dots, t_2])$.

\subsection{Parametrizations of the surfaces} In this section we outline the 
parametrizations for the three surfaces we consider.
\subsubsection{Veronese surface} We get a parametrization of a Veronese surface 
by the map associated to the linear system of all conics in $\mathbb{P}^2$:
{\footnotesize
\begin{verbatim}
i1 : k = ZZ/31;
i2 : P2 = k[t_0..t_2];
i3 : P5 = k[x_0..x_5];
i4 : veroneseMap=map(P2,P5,gens (ideal vars P2)^2);
o4 : RingMap P2 <--- P5
\end{verbatim}
} \noindent 
The image of our parametrization is defined 
by the ideal of $2\times 2$ minors of the matrix:
$$ 
\bgroup\begin{pmatrix}{x}_{0}&
      {x}_{1}&
      {x}_{2}\\
      {x}_{1}&
      {x}_{3}&
      {x}_{4}\\
      {x}_{2}&
      {x}_{4}&
      {x}_{5}\\
      \end{pmatrix}\egroup .
$$ 

\subsubsection{Cubic rational normal scroll} 
From the definition of rational normal scroll 
surface $S(a,b)\subset\mathbb{P}^{a+b+1}$, $0\leq a\leq b$,
we deduce an obvious parametrization 
$\mathbb{P}^2\dashrightarrow\mathbb{P}^1\times\mathbb{P}^1\dashrightarrow S(a,b)\subset\mathbb{P}^{a+b+1}$:
{\footnotesize 
\begin{verbatim}
i5 : scrollMap = (a,b) -> ( 
      t_0^(b-a+1)*(for j to a list t_0^(a-j)*t_1^j) | t_2*(for j to b list t_0^(b-j)*t_1^j) 
      ); 
\end{verbatim} 
} \noindent 
The image of this parametrization is defined 
by the ideal of $2\times 2$ minors of the matrix:
$$ 
\bgroup\begin{pmatrix}{x}_{0}&
      \cdots&
      {x}_{a-1}&
      {x}_{a+1}&
      \cdots&
      {x}_{a+b}\\
      {x}_{1}&
      \cdots&
      {x}_{a}&
      {x}_{a+2}&
      \cdots&
      {x}_{a+b+1}
      \end{pmatrix}\egroup .
$$ 
In particular, for $(a,b)=(1,2)$,
 we obtain the map 
$\mathbb{P}^2\dashrightarrow S(1,2)\subset \mathbb{P}^4=V(x_5)\subset \mathbb{P}^5$:
{\footnotesize
\begin{verbatim}
i6 : S12Map=map(P2,P5,scrollMap(1,2) | {0}); 
o6 : RingMap P2 <--- P5
\end{verbatim}
} \noindent 
Let us construct the cubic scroll surface in different ways. A first one given by a rational map from $\mathbb{P}^2$ to $\mathbb{P}^4$ and then embedded it in $\mathbb{P}^5$ as follows:
{\footnotesize 
\begin{verbatim}
i7 : needsPackage "SpecialFanoFourfolds";
i8 : cubicScrollMap = map(P2, P5, {t_0^2, t_0*t_1, t_1^2, t_2*t_0, t_1*t_2,0})
i9 : J = kernel cubicScrollMap
i10 : S12 = projectiveVariety J
\end{verbatim}
} \noindent 
The second way to construct the cubic scroll surface is to see it as the blow up of $\mathbb{P}^2$ in [0:0:1]. For that, we consider the Segre embedding 

 $$ \begin{aligned}
     \mathbb{P}^1 \times \mathbb{P}^2 & \longrightarrow \mathbb{P}^5 \\
    [u:v]\times[t_0:t_1:t_2]& \longrightarrow [t_0u:t_1u:t_2u,t_0v,t_1v,t_2v]
  \end{aligned};$$
the cubic scroll surface is the intersection of the image of $\mathbb{P}^1 \times \mathbb{P}^2$ via the embedding with the hyperplane $V(x_0-x_4)$ in $\mathbb{P}^5$. Its exceptional divisor is $[u:v]\times[0:0:1]$.
{\footnotesize 
\begin{verbatim}
i11 : P12 = k[u,v,t_0..t_2]
i12 : SegreMap = map(P12, P5, {t_0*u, t_1*u, t_2*u, t_0*v, t_1*v,t_2*v})
i13 : I = kernel SegreMap
i14 : H=ideal(x_0-x_4)
i15 : S12=ideal(H,I)
\end{verbatim}
} \noindent 

The key to find explicit examples of cubic fourfold in each component of the intersection $\mathcal{C}_8 \cap \mathcal{C}_i$ with $i=12,\ 20 $,  is to construct ``by hand" a plane $P$ that has the needed intersection propriety with the class of surface $S_i$ defining $\mathcal{C}_i$. Then  find a smooth cubic hypersurface containing the constructed $P$ and $S_i$. For that, we need the following method for later use.

{\footnotesize 
\begin{verbatim}
i14 : randomElementsFromIdeal =  method(TypicalValue => Ideal)
randomElementsFromIdeal(List, Ideal) := Ideal => (L,I)->(
trim ideal((gens I)*random(source gens I, (ring I)^(-L))))
\end{verbatim}
} \noindent 
A usual trick to find the appropriate plane  $P \subset \mathbb{P}^5$ with the needed intersection property with a surface $S_i$ is by using the previous method \textit{P = randomElementsFromIdeal(\{1,1,1\},Ideal)}
and varying the \textit{Ideal} from \textit{point ideal(Si)} to \textit{intersect(point ideal(Si), point ideal(Si))} which gives a priori an idea about the intersection $P.S_i$. Here, \textit{point ideal(Si)} picks a random rational point in $Si$. Let us give an example by constructing a plane $P$ that intersects the cubic scroll $S12$, defined previously by the ideal $J$, in one point.
{\footnotesize 
\begin{verbatim}
i16 : a = point J
i17 : p = randomElementsFromIdeal({1,1,1},a)
i18 : P = projectiveVariety p
\end{verbatim}
} \noindent
Once the plane is constructed, one needs to define a smooth cubic hypersurface containing the plane and the cubic scroll as follows:
{\footnotesize 
\begin{verbatim}
i19 : X=random(3,S12+P)
\end{verbatim}
} \noindent

\medskip 
Using this method one can find explicit examples of cubic fourfolds in each component of the intersection $\mathcal{C}_{8} \cap \mathcal{C}_i$, for $i=12,\ 20$ as follows.

\begin{ex}\label{exscroll}

Here is the collection of examples of  cubics containing a cubic scroll surface and a plane.
\begin{enumerate}
 \item $\Pi_1=(x_{0},\,x_{1},\,x_{2})$; $\Pi_1 \cap S$ is a line; $C =  -6\,x_{0}x_{1}^{2}-x_{1}^{3}+6\,x_{0}^{2}x_{2}+x_{0}x_{1}x_{2}+8\,x_{1}^{2}x_{2}-8\,x_{0}x_{2}^{2}+15\,x_{0}x_{1}x_{3}-2\,x_{1}^{2}x_{3}+12\,x_{0}x_{2}x_{3}-2\,x_{1}x_{2}x_{3}-8\,x_{2}^{2}x_{3}-10\,x_{1}x_{3}^{2}-6\,x_{2}x_{3}^{2}-15\,x_{0}^{2}x_{4}-10\,x_{0}x_{1}x_{4}-9\,x_{1}^{2}x_{4}+11\,x_{0}x_{2}x_{4}+8\,x_{1}x_{2}x_{4}+10\,x_{0}x_{3}x_{4}-15\,x_{1}x_{3}x_{4}-13\,x_{2}x_{3}x_{4}-10\,x_{0}x_{4}^{2}+13\,x_{1}x_{4}^{2}-9\,x_{0}^{2}x_{5}-12\,x_{0}x_{1}x_{5}+8\,x_{1}^{2}x_{5}+15\,x_{0}x_{2}x_{5}-9\,x_{1}x_{2}x_{5}-7\,x_{2}^{2}x_{5}-14\,x_{0}x_{3}x_{5}+5\,x_{1}x_{3}x_{5}-4\,x_{2}x_{3}x_{5}-13\,x_{0}x_{4}x_{5}-11\,x_{1}x_{4}x_{5}+9\,x_{2}x_{4}x_{5}+8\,x_{0}x_{5}^{2}+2\,x_{1}x_{5}^{2}-13\,x_{2}x_{5}^{2}$;
    \item $\Pi_2=(x_{3},\,x_{4},\,x_{5})$; $\Pi_2 \cap S$ is a smooth conic; $C= -6\,x_{0}x_{1}x_{3}+6\,x_{1}^{2}x_{3}-13\,x_{0}x_{2}x_{3}-2\,x_{1}x_{2}x_{3}+13\,x_{2}^{2}x_{3}-9\,x_{1}x_{3}^{2}+8\,x_{2}x_{3}^{2}+6\,x_{0}^{2}x_{4}+7\,x_{0}x_{1}x_{4}+x_{1}^{2}x_{4}+x_{0}x_{2}x_{4}-13\,x_{1}x_{2}x_{4}+9\,x_{0}x_{3}x_{4}-x_{1}x_{3}x_{4}-8\,x_{2}x_{3}x_{4}-7\,x_{0}x_{4}^{2}+8\,x_{1}x_{4}^{2}-4\,x_{0}^{2}x_{5}-12\,x_{0}x_{1}x_{5}-12\,x_{1}^{2}x_{5}+x_{0}x_{2}x_{5}+4\,x_{1}x_{2}x_{5}+15\,x_{2}^{2}x_{5}-15\,x_{0}x_{3}x_{5}+4\,x_{1}x_{3}x_{5}+5\,x_{2}x_{3}x_{5}-15\,x_{3}^{2}x_{5}+10\,x_{0}x_{4}x_{5}+5\,x_{1}x_{4}x_{5}+6\,x_{2}x_{4}x_{5}-6\,x_{3}x_{4}x_{5}+5\,x_{4}^{2}x_{5}+3\,x_{0}x_{5}^{2}+x_{1}x_{5}^{2}+14\,x_{2}x_{5}^{2}-5\,x_{3}x_{5}^{2}-5\,x_{4}x_{5}^{2}+6\,x_{5}^{3}$;
   
    \item $P=V(x_{2}-x_{3}+7\,x_{4}+4\,x_{5},\,x_{1}+15\,x_{3}-14\,x_{4}-9\,x_{5},\,x_{0}-14\,x_{3}+9\,x_{5})$; $P \cap S$ is empty; $C =  -14\,x_{0}x_{1}^{2}+11\,x_{1}^{3}+14\,x_{0}^{2}x_{2}-11\,x_{0}x_{1}x_{2}-3\,x_{1}^{2}x_{2}+3\,x_{0}x_{2}^{2}+6\,x_{0}x_{1}x_{3}+13\,x_{0}x_{2}x_{3}-3\,x_{1}x_{2}x_{3}+7\,x_{2}^{2}x_{3}+8\,x_{1}x_{3}^{2}-15\,x_{2}x_{3}^{2}-6\,x_{0}^{2}x_{4}-13\,x_{0}x_{1}x_{4}+11\,x_{1}^{2}x_{4}-8\,x_{0}x_{2}x_{4}-7\,x_{1}x_{2}x_{4}-8\,x_{0}x_{3}x_{4}+6\,x_{1}x_{3}x_{4}-13\,x_{2}x_{3}x_{4}+9\,x_{0}x_{4}^{2}+13\,x_{1}x_{4}^{2}-3\,x_{0}^{2}x_{5}-4\,x_{0}x_{1}x_{5}-4\,x_{1}^{2}x_{5}-10\,x_{0}x_{2}x_{5}-14\,x_{1}x_{2}x_{5}+12\,x_{2}^{2}x_{5}-8\,x_{0}x_{3}x_{5}+10\,x_{1}x_{3}x_{5}+2\,x_{2}x_{3}x_{5}+12\,x_{3}^{2}x_{5}+14\,x_{0}x_{4}x_{5}-10\,x_{1}x_{4}x_{5}-2\,x_{2}x_{4}x_{5}+2\,x_{3}x_{4}x_{5}-9\,x_{4}^{2}x_{5}+9\,x_{0}x_{5}^{2}-11\,x_{1}x_{5}^{2}-2\,x_{2}x_{5}^{2}-13\,x_{3}x_{5}^{2}-14\,x_{4}x_{5}^{2}-15\,x_{5}^{3}$;
    \item $P=V(x_{2}+7\,x_{3}-10\,x_{4}+4\,x_{5},\,x_{1}+5\,x_{3}-11\,x_{4}-3\,x_{5},\,x_{0}-6\,x_{3}+9\,x_{4}-10\,x_{5})$; $P \cap S$ consists of 1 point; $C =  12\,x_{0}x_{1}^{2}-6\,x_{1}^{3}-12\,x_{0}^{2}x_{2}+6\,x_{0}x_{1}x_{2}+15\,x_{1}^{2}x_{2}-15\,x_{0}x_{2}^{2}+2\,x_{0}x_{1}x_{3}-7\,x_{1}^{2}x_{3}+5\,x_{0}x_{2}x_{3}-11\,x_{1}x_{2}x_{3}-10\,x_{2}^{2}x_{3}-6\,x_{1}x_{3}^{2}+4\,x_{2}x_{3}^{2}-2\,x_{0}^{2}x_{4}+2\,x_{0}x_{1}x_{4}+15\,x_{1}^{2}x_{4}-4\,x_{0}x_{2}x_{4}+10\,x_{1}x_{2}x_{4}+6\,x_{0}x_{3}x_{4}+11\,x_{1}x_{3}x_{4}-14\,x_{2}x_{3}x_{4}-15\,x_{0}x_{4}^{2}+14\,x_{1}x_{4}^{2}+8\,x_{0}^{2}x_{5}-11\,x_{0}x_{1}x_{5}+3\,x_{1}^{2}x_{5}+4\,x_{1}x_{2}x_{5}-x_{2}^{2}x_{5}-5\,x_{0}x_{3}x_{5}-2\,x_{1}x_{3}x_{5}-9\,x_{2}x_{3}x_{5}-8\,x_{0}x_{4}x_{5}+10\,x_{1}x_{4}x_{5}-11\,x_{2}x_{4}x_{5}-12\,x_{3}x_{4}x_{5}-14\,x_{4}^{2}x_{5}+6\,x_{0}x_{5}^{2}+14\,x_{1}x_{5}^{2}+7\,x_{2}x_{5}^{2}+10\,x_{3}x_{5}^{2}+3\,x_{4}x_{5}^{2}+8\,x_{5}^{3}$;
    \item $P= V(x_{2}+15\,x_{4}-13\,x_{5},\,x_{1}+15\,x_{3}+3\,x_{5},\,x_{0}-15\,x_{3}-10\,x_{4}+14\,x_{5})$; $P \cap S$ consists of 2 points; $C=2\,x_{0}x_{1}^{2}+7\,x_{1}^{3}-2\,x_{0}^{2}x_{2}-7\,x_{0}x_{1}x_{2}-6\,x_{1}^{2}x_{2}+6\,x_{0}x_{2}^{2}+x_{0}x_{1}x_{3}+11\,x_{1}^{2}x_{3}+11\,x_{0}x_{2}x_{3}-8\,x_{1}x_{2}x_{3}+11\,x_{2}^{2}x_{3}-14\,x_{1}x_{3}^{2}-11\,x_{2}x_{3}^{2}-x_{0}^{2}x_{4}+9\,x_{0}x_{1}x_{4}-12\,x_{1}^{2}x_{4}-11\,x_{0}x_{2}x_{4}-11\,x_{1}x_{2}x_{4}+14\,x_{0}x_{3}x_{4}-13\,x_{1}x_{3}x_{4}+2\,x_{2}x_{3}x_{4}-7\,x_{0}x_{4}^{2}-2\,x_{1}x_{4}^{2}+9\,x_{0}^{2}x_{5}+7\,x_{0}x_{1}x_{5}+8\,x_{1}^{2}x_{5}-6\,x_{0}x_{2}x_{5}+3\,x_{1}x_{2}x_{5}+11\,x_{2}^{2}x_{5}-14\,x_{0}x_{3}x_{5}+10\,x_{1}x_{3}x_{5}-7\,x_{2}x_{3}x_{5}+8\,x_{3}^{2}x_{5}+5\,x_{0}x_{4}x_{5}-6\,x_{1}x_{4}x_{5}-13\,x_{2}x_{4}x_{5}-x_{3}x_{4}x_{5}+3\,x_{4}^{2}x_{5}-x_{0}x_{5}^{2}+2\,x_{1}x_{5}^{2}+5\,x_{2}x_{5}^{2}+x_{3}x_{5}^{2}+10\,x_{4}x_{5}^{2}$;
    \item $P = V(x_{5},\,x_{1}+12\,x_{2}-12\,x_{3}-4\,x_{4},\,x_{0}-7\,x_{2}+14\,x_{3}-3\,x_{4})$; $P \cap Z$ consists of 3 points;
    $C=  6\,x_{0}x_{1}^{2}+5\,x_{1}^{3}-6\,x_{0}^{2}x_{2}-5\,x_{0}x_{1}x_{2}-13\,x_{1}^{2}x_{2}+13\,x_{0}x_{2}^{2}-5\,x_{0}x_{1}x_{3}+15\,x_{1}^{2}x_{3}+13\,x_{0}x_{2}x_{3}+7\,x_{1}x_{2}x_{3}-12\,x_{2}^{2}x_{3}+7\,x_{1}x_{3}^{2}-8\,x_{2}x_{3}^{2}+5\,x_{0}^{2}x_{4}+3\,x_{0}x_{1}x_{4}+6\,x_{1}^{2}x_{4}-13\,x_{0}x_{2}x_{4}+12\,x_{1}x_{2}x_{4}-7\,x_{0}x_{3}x_{4}-3\,x_{1}x_{3}x_{4}+3\,x_{2}x_{3}x_{4}+11\,x_{0}x_{4}^{2}-3\,x_{1}x_{4}^{2}-9\,x_{0}^{2}x_{5}+2\,x_{0}x_{1}x_{5}+7\,x_{1}^{2}x_{5}+14\,x_{0}x_{2}x_{5}+13\,x_{1}x_{2}x_{5}-5\,x_{2}^{2}x_{5}-3\,x_{0}x_{3}x_{5}+14\,x_{1}x_{3}x_{5}-8\,x_{2}x_{3}x_{5}+8\,x_{3}^{2}x_{5}+3\,x_{0}x_{4}x_{5}+13\,x_{1}x_{4}x_{5}-2\,x_{2}x_{4}x_{5}+8\,x_{3}x_{4}x_{5}-7\,x_{4}^{2}x_{5}-9\,x_{0}x_{5}^{2}-5\,x_{1}x_{5}^{2}+11\,x_{2}x_{5}^{2}+5\,x_{3}x_{5}^{2}+3\,x_{4}x_{5}^{2}+14\,x_{5}^{3}$; 
    \item $P=V(x_{3}-x_{4}-2\,x_{5},\,x_{1}-x_{2}-14\,x_{5},\,x_{0}-x_{2}-2\,x_{5})$; $P \cap S$ is a line of the ruling of the scroll, i.e. the image of a line passing through $[0:0:1]$ via the map of  $i8$ ; $C=6\,x_{0}x_{1}^{2}+7\,x_{1}^{3}-6\,x_{0}^{2}x_{2}-7\,x_{0}x_{1}x_{2}-15\,x_{1}^{2}x_{2}+15\,x_{0}x_{2}^{2}-14\,x_{0}x_{1}x_{3}+2\,x_{1}^{2}x_{3}+3\,x_{0}x_{2}x_{3}+15\,x_{1}x_{2}x_{3}-7\,x_{2}^{2}x_{3}-12\,x_{1}x_{3}^{2}+4\,x_{2}x_{3}^{2}+14\,x_{0}^{2}x_{4}-5\,x_{0}x_{1}x_{4}+5\,x_{1}^{2}x_{4}+11\,x_{0}x_{2}x_{4}+7\,x_{1}x_{2}x_{4}+12\,x_{0}x_{3}x_{4}+12\,x_{1}x_{3}x_{4}-14\,x_{2}x_{3}x_{4}+15\,x_{0}x_{4}^{2}+14\,x_{1}x_{4}^{2}-12\,x_{0}^{2}x_{5}+3\,x_{0}x_{1}x_{5}+15\,x_{1}^{2}x_{5}-13\,x_{0}x_{2}x_{5}-11\,x_{1}x_{2}x_{5}+10\,x_{2}^{2}x_{5}-2\,x_{0}x_{3}x_{5}-15\,x_{1}x_{3}x_{5}-15\,x_{2}x_{3}x_{5}-9\,x_{3}^{2}x_{5}+8\,x_{0}x_{4}x_{5}-12\,x_{1}x_{4}x_{5}+12\,x_{2}x_{4}x_{5}-9\,x_{3}x_{4}x_{5}-15\,x_{4}^{2}x_{5}-x_{0}x_{5}^{2}+5\,x_{1}x_{5}^{2}-15\,x_{2}x_{5}^{2}-8\,x_{3}x_{5}^{2}+6\,x_{4}x_{5}^{2}+15\,x_{5}^{3}$; 
    \item $P=V(x_{3}+11\,x_{4},\,x_{1}-7\,x_{4},\,x_{0}+8\,x_{4})$; $P \cap S$ is the directrix line of the cubic scroll surface; the directrix line is the image of the exceptional divisor, which is easily obtained via the Segre embedding of $\PP^1\times \PP^2$ as in $i12$  here above; $C=-12\,x_{0}^{3}+10\,x_{0}^{2}x_{1}-10\,x_{0}x_{1}^{2}+11\,x_{0}^{2}x_{2}-10\,x_{0}x_{1}x_{2}-2\,x_{0}^{2}x_{3}+3\,x_{0}x_{1}x_{3}-15\,x_{1}^{2}x_{3}+6\,x_{0}x_{2}x_{3}-5\,x_{1}x_{2}x_{3}+8\,x_{2}^{2}x_{3}-15\,x_{0}x_{3}^{2}-14\,x_{1}x_{3}^{2}-11\,x_{2}x_{3}^{2}-x_{0}^{2}x_{4}-9\,x_{0}x_{1}x_{4}+10\,x_{1}^{2}x_{4}+x_{0}x_{2}x_{4}+15\,x_{1}x_{2}x_{4}-5\,x_{2}^{2}x_{4}-11\,x_{0}x_{3}x_{4}-3\,x_{1}x_{3}x_{4}-12\,x_{2}x_{3}x_{4}+15\,x_{3}^{2}x_{4}-15\,x_{0}x_{4}^{2}+14\,x_{1}x_{4}^{2}+10\,x_{2}x_{4}^{2}-4\,x_{3}x_{4}^{2}-3\,x_{4}^{3}+10\,x_{0}^{2}x_{5}-6\,x_{0}x_{1}x_{5}-5\,x_{1}^{2}x_{5}-12\,x_{0}x_{2}x_{5}+5\,x_{1}x_{2}x_{5}-9\,x_{0}x_{3}x_{5}+11\,x_{1}x_{3}x_{5}+3\,x_{2}x_{3}x_{5}-15\,x_{0}x_{4}x_{5}-11\,x_{1}x_{4}x_{5}-5\,x_{2}x_{4}x_{5}-11\,x_{3}x_{4}x_{5}-14\,x_{0}x_{5}^{2}+9\,x_{1}x_{5}^{2}+11\,x_{4}x_{5}^{2}$. See the discussion in Sect. \ref{tripleint} for more details on this kind of cubics.

\end{enumerate}
   
\end{ex}   

\begin{ex}\label{exveronese}  Smooth cubics containing a Veronese surface and a plane

\begin{enumerate}
    \item $\Pi_1=V(x_{0},\,x_{1},\,x_{2})$; $\Pi_1 \cap Z$ is a smooth conic of genus 0; $C =  2\,x_{0}x_{1}^{2}-6\,x_{1}^{3}+8\,x_{0}x_{1}x_{2}-3\,x_{1}^{2}x_{2}-4\,x_{0}x_{2}^{2}-12\,x_{1}x_{2}^{2}-15\,x_{2}^{3}-2\,x_{0}^{2}x_{3}+6\,x_{0}x_{1}x_{3}+x_{1}^{2}x_{3}-7\,x_{0}x_{2}x_{3}-x_{1}x_{2}x_{3}-8\,x_{2}^{2}x_{3}-x_{0}x_{3}^{2}+13\,x_{2}x_{3}^{2}-8\,x_{0}^{2}x_{4}+10\,x_{0}x_{1}x_{4}-11\,x_{1}^{2}x_{4}+12\,x_{0}x_{2}x_{4}+12\,x_{1}x_{2}x_{4}+3\,x_{2}^{2}x_{4}+12\,x_{0}x_{3}x_{4}-13\,x_{1}x_{3}x_{4}-x_{2}x_{3}x_{4}-13\,x_{0}x_{4}^{2}-6\,x_{1}x_{4}^{2}-10\,x_{2}x_{4}^{2}+4\,x_{0}^{2}x_{5}+10\,x_{1}^{2}x_{5}+15\,x_{0}x_{2}x_{5}+3\,x_{1}x_{2}x_{5}+6\,x_{2}^{2}x_{5}-x_{0}x_{3}x_{5}+7\,x_{1}x_{3}x_{5}-15\,x_{2}x_{3}x_{5}-6\,x_{0}x_{4}x_{5}-6\,x_{1}x_{4}x_{5}-6\,x_{2}x_{4}x_{5}-6\,x_{0}x_{5}^{2}+6\,x_{1}x_{5}^{2}$;
    \item $P=  V(x_{2}-6\,x_{3}-2\,x_{4}-14\,x_{5},\,x_{1}+8\,x_{3}-10\,x_{4}+11\,x_{5},\,x_{0}+11\,x_{3}-11\,x_{4}-9\,x_{5})$; $P \cap Z$ is empty; $C =7\,x_{0}x_{1}^{2}+12\,x_{1}^{3}+2\,x_{0}x_{1}x_{2}-10\,x_{1}^{2}x_{2}-6\,x_{0}x_{2}^{2}+6\,x_{1}x_{2}^{2}+10\,x_{2}^{3}-7\,x_{0}^{2}x_{3}-12\,x_{0}x_{1}x_{3}-12\,x_{1}^{2}x_{3}-12\,x_{0}x_{2}x_{3}-11\,x_{1}x_{2}x_{3}+12\,x_{0}x_{3}^{2}-4\,x_{2}x_{3}^{2}-2\,x_{0}^{2}x_{4}-9\,x_{0}x_{1}x_{4}+2\,x_{1}^{2}x_{4}-8\,x_{0}x_{2}x_{4}-5\,x_{1}x_{2}x_{4}+9\,x_{0}x_{3}x_{4}+4\,x_{1}x_{3}x_{4}+9\,x_{2}x_{3}x_{4}+10\,x_{0}x_{4}^{2}-2\,x_{1}x_{4}^{2}+11\,x_{2}x_{4}^{2}-x_{3}x_{4}^{2}-12\,x_{4}^{3}+6\,x_{0}^{2}x_{5}+2\,x_{0}x_{1}x_{5}+10\,x_{1}^{2}x_{5}-10\,x_{0}x_{2}x_{5}-10\,x_{1}x_{2}x_{5}+6\,x_{2}^{2}x_{5}-15\,x_{0}x_{3}x_{5}-7\,x_{1}x_{3}x_{5}+11\,x_{2}x_{3}x_{5}+x_{3}^{2}x_{5}+10\,x_{0}x_{4}x_{5}+9\,x_{1}x_{4}x_{5}+9\,x_{2}x_{4}x_{5}+12\,x_{3}x_{4}x_{5}-4\,x_{4}^{2}x_{5}-6\,x_{0}x_{5}^{2}-9\,x_{1}x_{5}^{2}+4\,x_{3}x_{5}^{2}$;
   
    \item $P = V(x_{2}+13\,x_{3}+4\,x_{4}-10\,x_{5},\,x_{1}-5\,x_{3}+12\,x_{4}-10\,x_{5},\,x_{0}-11\,x_{3}-7\,x_{4}+15\,x_{5})$; $P \cap Z$ consists of 1 point; $C =-8\,x_{0}x_{1}^{2}+8\,x_{1}^{3}+13\,x_{0}x_{1}x_{2}+10\,x_{1}^{2}x_{2}-14\,x_{0}x_{2}^{2}-4\,x_{1}x_{2}^{2}-13\,x_{2}^{3}+8\,x_{0}^{2}x_{3}-8\,x_{0}x_{1}x_{3}+13\,x_{1}^{2}x_{3}-8\,x_{0}x_{2}x_{3}+12\,x_{1}x_{2}x_{3}-11\,x_{2}^{2}x_{3}-13\,x_{0}x_{3}^{2}+12\,x_{2}x_{3}^{2}-13\,x_{0}^{2}x_{4}-2\,x_{0}x_{1}x_{4}-14\,x_{1}^{2}x_{4}+12\,x_{0}x_{2}x_{4}+2\,x_{1}x_{2}x_{4}+9\,x_{2}^{2}x_{4}+2\,x_{0}x_{3}x_{4}-12\,x_{1}x_{3}x_{4}-14\,x_{2}x_{3}x_{4}-13\,x_{0}x_{4}^{2}+9\,x_{1}x_{4}^{2}+12\,x_{2}x_{4}^{2}-9\,x_{3}x_{4}^{2}+5\,x_{4}^{3}+14\,x_{0}^{2}x_{5}-8\,x_{0}x_{1}x_{5}+2\,x_{1}^{2}x_{5}+13\,x_{0}x_{2}x_{5}-6\,x_{1}x_{2}x_{5}-5\,x_{2}^{2}x_{5}-11\,x_{0}x_{3}x_{5}+5\,x_{1}x_{3}x_{5}+4\,x_{2}x_{3}x_{5}+9\,x_{3}^{2}x_{5}-3\,x_{0}x_{4}x_{5}+15\,x_{1}x_{4}x_{5}-4\,x_{2}x_{4}x_{5}-5\,x_{3}x_{4}x_{5}+14\,x_{4}^{2}x_{5}+5\,x_{0}x_{5}^{2}+4\,x_{1}x_{5}^{2}-14\,x_{3}x_{5}^{2}$;
    \item $P =  V(x_{2}+12\,x_{3}+2\,x_{4}-7\,x_{5},\,x_{1}-x_{3}-5\,x_{5},\,x_{0}+7\,x_{3}-4\,x_{4}+10\,x_{5})$; $P \cap Z$ consists of 2 points; $C =  2\,x_{0}x_{1}^{2}+9\,x_{1}^{3}-7\,x_{0}x_{1}x_{2}+x_{1}^{2}x_{2}+15\,x_{0}x_{2}^{2}+x_{1}x_{2}^{2}+3\,x_{2}^{3}-2\,x_{0}^{2}x_{3}-9\,x_{0}x_{1}x_{3}+4\,x_{1}^{2}x_{3}-x_{0}x_{2}x_{3}-2\,x_{1}x_{2}x_{3}-5\,x_{2}^{2}x_{3}-4\,x_{0}x_{3}^{2}-8\,x_{2}x_{3}^{2}+7\,x_{0}^{2}x_{4}-2\,x_{1}^{2}x_{4}+10\,x_{0}x_{2}x_{4}+13\,x_{1}x_{2}x_{4}+x_{2}^{2}x_{4}+4\,x_{0}x_{3}x_{4}+8\,x_{1}x_{3}x_{4}+5\,x_{2}x_{3}x_{4}+12\,x_{0}x_{4}^{2}-x_{1}x_{4}^{2}-6\,x_{2}x_{4}^{2}+6\,x_{3}x_{4}^{2}-2\,x_{4}^{3}-15\,x_{0}^{2}x_{5}-11\,x_{0}x_{1}x_{5}-2\,x_{1}^{2}x_{5}-3\,x_{0}x_{2}x_{5}-8\,x_{2}^{2}x_{5}+13\,x_{0}x_{3}x_{5}-4\,x_{1}x_{3}x_{5}-14\,x_{2}x_{3}x_{5}-6\,x_{3}^{2}x_{5}-x_{0}x_{4}x_{5}-11\,x_{1}x_{4}x_{5}+7\,x_{2}x_{4}x_{5}+2\,x_{3}x_{4}x_{5}+12\,x_{4}^{2}x_{5}+,x_{1}x_{3}x_{5}-7\,x_{2}x_{3}x_{5}-2\,x_{0}x_{4}x_{5}+3\,x_{0}x_{5}^{2}$;
    \item $P=  V(x_{2}-13\,x_{3}-3\,x_{4}-13\,x_{5},\,x_{1}+9\,x_{3}+5\,x_{4}+14\,x_{5},\,x_{0}+x_{3}-4\,x_{4}-5\,x_{5})$; $P \cap Z$ consists of 3 points; $C =  -10\,x_{0}x_{1}^{2}+15\,x_{1}^{3}-13\,x_{0}x_{1}x_{2}+10\,x_{1}^{2}x_{2}-7\,x_{0}x_{2}^{2}-3\,x_{1}x_{2}^{2}+9\,x_{2}^{3}+10\,x_{0}^{2}x_{3}-15\,x_{0}x_{1}x_{3}+15\,x_{1}^{2}x_{3}-4\,x_{1}x_{2}x_{3}+7\,x_{2}^{2}x_{3}-15\,x_{0}x_{3}^{2}-5\,x_{2}x_{3}^{2}+13\,x_{0}^{2}x_{4}-10\,x_{0}x_{1}x_{4}-x_{1}^{2}x_{4}+3\,x_{0}x_{2}x_{4}+6\,x_{1}x_{2}x_{4}-15\,x_{2}^{2}x_{4}+5\,x_{0}x_{3}x_{4}+5\,x_{1}x_{3}x_{4}+x_{1}x_{4}^{2}+8\,x_{2}x_{4}^{2}-3\,x_{3}x_{4}^{2}+14\,x_{4}^{3}+7\,x_{0}^{2}x_{5}-9\,x_{1}^{2}x_{5}-9\,x_{0}x_{2}x_{5}+2\,x_{1}x_{2}x_{5}-12\,x_{2}^{2}x_{5}-4\,x_{0}x_{3}x_{5}-x_{1}x_{3}x_{5}-x_{2}x_{3}x_{5}+3\,x_{3}^{2}x_{5}+13\,x_{0}x_{4}x_{5}-7\,x_{1}x_{4}x_{5}-14\,x_{2}x_{4}x_{5}-14\,x_{3}x_{4}x_{5}+11\,x_{4}^{2}x_{5}+12\,x_{0}x_{5}^{2}+14\,x_{1}x_{5}^{2}-11\,x_{3}x_{5}^{2}$;
    \item The Veronese surface can degenerate into $Z=S(1,2) \cup P$, union of a cubic scroll surface $S(1,2)$ and a plane such that $P \cap S(1,2)$ is the directrix line of the scroll; in this case $\gamma = P\cdot (S(1,2)\cup P)=4$ (see Sect. \ref{tripleint}). A cubic fourfold in $\mathcal{C}_{20}$ containing this degeneration of a Veronese surface is automatically contained in $\mathcal{C}_{8} \cap \mathcal{C}_{12}$; $C=-12\,x_{0}^{3}+10\,x_{0}^{2}x_{1}-10\,x_{0}x_{1}^{2}+11\,x_{0}^{2}x_{2}-10\,x_{0}x_{1}x_{2}-2\,x_{0}^{2}x_{3}+3\,x_{0}x_{1}x_{3}-15\,x_{1}^{2}x_{3}+6\,x_{0}x_{2}x_{3}-5\,x_{1}x_{2}x_{3}+8\,x_{2}^{2}x_{3}-15\,x_{0}x_{3}^{2}-14\,x_{1}x_{3}^{2}-11\,x_{2}x_{3}^{2}-x_{0}^{2}x_{4}-9\,x_{0}x_{1}x_{4}+10\,x_{1}^{2}x_{4}+x_{0}x_{2}x_{4}+15\,x_{1}x_{2}x_{4}-5\,x_{2}^{2}x_{4}-11\,x_{0}x_{3}x_{4}-3\,x_{1}x_{3}x_{4}-12\,x_{2}x_{3}x_{4}+15\,x_{3}^{2}x_{4}-15\,x_{0}x_{4}^{2}+14\,x_{1}x_{4}^{2}+10\,x_{2}x_{4}^{2}-4\,x_{3}x_{4}^{2}-3\,x_{4}^{3}+10\,x_{0}^{2}x_{5}-6\,x_{0}x_{1}x_{5}-5\,x_{1}^{2}x_{5}-12\,x_{0}x_{2}x_{5}+5\,x_{1}x_{2}x_{5}-9\,x_{0}x_{3}x_{5}+11\,x_{1}x_{3}x_{5}+3\,x_{2}x_{3}x_{5}-15\,x_{0}x_{4}x_{5}-11\,x_{1}x_{4}x_{5}-5\,x_{2}x_{4}x_{5}-11\,x_{3}x_{4}x_{5}-14\,x_{0}x_{5}^{2}+9\,x_{1}x_{5}^{2}+11\,x_{4}x_{5}^{2}$.

\end{enumerate}

\end{ex}

\bibliography{bib_tocho}

\begin{thebibliography}{10}

\bibitem{AHTVA}
N.~Addington, B.~Hassett, Y.~Tschinkel, and A.~V{\'a}rilly-Alvarado.
\newblock Cubic fourfolds fibered in sextic del {Pezzo} surfaces.
\newblock {\em Am. J. Math.}, 141(6):1479--1500, 2019.

\bibitem{ABB}
A.~Auel, M.~Bernardara, and M.~Bolognesi.
\newblock Fibrations in complete intersections of quadrics, {Clifford}
  algebras, derived categories, and rationality problems.
\newblock {\em J. Math. Pures Appl. (9)}, 102(1):249--291, 2014.

\bibitem{ABBVA}
A.~Auel, M.~Bernardara, M.~Bolognesi, and A.~V{\'a}rilly-Alvarado.
\newblock Cubic fourfolds containing a plane and a quintic del {Pezzo} surface.
\newblock {\em Algebr. Geom.}, 1(2):181--193, 2014.

\bibitem{awada}
H.~Awada.
\newblock Rational fibered cubic fourfolds.
\newblock {\em Manuscr. Math.}, 175(3-4):1003--1034, 2024.

\bibitem{BD}
A.~Beauville and R.~Donagi.
\newblock La vari{\'e}t{\'e} des droites d'une hypersurface cubique de
  dimension 4. ({The} variety of lines of a cubic hypersurface of dimension 4).
\newblock {\em C. R. Acad. Sci., Paris, S{\'e}r. I}, 301:703--706, 1985.

\bibitem{BP20}
M.~Bolognesi and C.~Pedrini.
\newblock The transcendental motive of a cubic fourfold.
\newblock {\em J. Pure Appl. Algebra}, 224(8):16, 2020.
\newblock Id/No 106333.

\bibitem{BRS}
M.~Bolognesi, F.~Russo, and G.~Staglian{\`o}.
\newblock Some loci of rational cubic fourfolds.
\newblock {\em Math. Ann.}, 373(1-2):165--190, 2019.

\bibitem{CMR}
C.~Ciliberto, M.~Mella, and F.~Russo.
\newblock Varieties with one apparent double point.
\newblock {\em J. Algebr. Geom.}, 13(3):475--512, 2004.

\bibitem{CR}
C.~Ciliberto and F.~Russo.
\newblock On the classification of {OADP} varieties.
\newblock {\em Sci. China, Math.}, 54(8):1561--1575, 2011.

\bibitem{CT}
J.-L. Colliot-Th{\'e}l{\`e}ne.
\newblock Galois descent on the second {Chow} group: development and
  applications.
\newblock {\em Doc. Math.}, Extra Vol.:195--220, 2015.

\bibitem{EGHP}
D.~Eisenbud, M.~Green, K.~Hulek, and S.~Popescu.
\newblock Restricting linear syzygies: algebra and geometry.
\newblock {\em Compos. Math.}, 141(6):1460--1478, 2005.

\bibitem{Fano}
G.~d. Fano.
\newblock Sulle forme cubiche dello spazio a cinque dimensioni contenenti
  rigate razionali del 4° ordine.
\newblock {\em Commentarii mathematici Helvetici}, 15:71--80, 1942/43.

\bibitem{Fu}
W.~Fulton.
\newblock {\em Intersection theory.}, volume~2 of {\em Ergeb. Math. Grenzgeb.,
  3. Folge}.
\newblock Berlin: Springer, 2nd ed. edition, 1998.

\bibitem{GR}
T.~Graber.
\newblock Veronese surfaces and line arrangements.
\newblock 2008.

\bibitem{macaulay2}
D.~R. Grayson and M.~E. Stillman.
\newblock Macaulay2, a software system for research in algebraic geometry.
\newblock Available at \url{https://macaulay2.com/}.

\bibitem{Ha}
B.~Hassett.
\newblock Some rational cubic fourfolds.
\newblock {\em J. Algebr. Geom.}, 8(1):103--114, 1999.

\bibitem{Ha2}
B.~Hassett.
\newblock Special cubic fourfolds.
\newblock {\em Compos. Math.}, 120(1):1--23, 2000.

\bibitem{KT}
M.~Kontsevich and Y.~Tschinkel.
\newblock Specialization of birational types.
\newblock {\em Invent. Math.}, 217(2):415--432, 2019.

\bibitem{laza}
R.~Laza.
\newblock The moduli space of cubic fourfolds via the period map.
\newblock {\em Ann. Math. (2)}, 172(1):673--711, 2010.

\bibitem{Loij}
E.~Looijenga.
\newblock The period map for cubic fourfolds.
\newblock {\em Invent. Math.}, 177(1):213--233, 2009.

\bibitem{RS1}
F.~Russo and G.~Staglian{\`o}.
\newblock Congruences of 5-secant conics and the rationality of some admissible
  cubic fourfolds.
\newblock {\em Duke Math. J.}, 168(5):849--865, 2019.

\bibitem{RS2}
F.~Russo and G.~Staglian{\`o}.
\newblock Trisecant flops, their associated {{\(K3\)}} surfaces and the
  rationality of some cubic fourfolds.
\newblock {\em J. Eur. Math. Soc. (JEMS)}, 25(6):2435--2482, 2023.

\bibitem{voisin}
C.~Voisin.
\newblock Th{\'e}or{\`e}me de {Torelli} pour les cubiques de {{\({\mathbb{P}}^
  5\)}}. (torelli theorem for the cubics of {{\({\mathbb{P}}^ 5)\)}}.
\newblock {\em Invent. Math.}, 86:577--601, 1986.

\bibitem{Voi}
C.~Voisin.
\newblock Some aspects of the {Hodge} conjecture.
\newblock {\em Jpn. J. Math. (3)}, 2(2):261--296, 2007.

\bibitem{YYu}
S.~Yang and X.~Yu.
\newblock Rational cubic fourfolds in {Hassett} divisors.
\newblock {\em C. R., Math., Acad. Sci. Paris}, 358(2):129--137, 2020.

\bibitem{yangyu}
S.~Yang and X.~Yu.
\newblock On lattice polarizable cubic fourfolds.
\newblock {\em Res. Math. Sci.}, 10(1):39, 2023.
\newblock Id/No 2.

\end{thebibliography}
\bibliographystyle{abbrv}
\end{document}